\documentclass[10 pt, leqno]{article}
\date{}
\usepackage{amssymb, amsbsy, amsmath, amsfonts, amssymb, amscd, mathrsfs, amsthm, dsfont}
\usepackage[english]{babel}
\usepackage[T1]{fontenc}
\usepackage{indentfirst}
\usepackage{color}

\makeatletter
\@addtoreset{equation}{section}
\makeatother

\newtheorem{statement}{}[section]
\newtheorem{theorem}[statement]{Theorem}
\newtheorem{lemma}[statement]{Lemma}

\newtheorem{proposition}[statement]{Proposition}
\newtheorem{definition}[statement]{Definition}
\newtheorem{corollary}[statement]{Corollary}

\newcommand\C{\mathbb C}

\newcommand\R{\mathbb R}
\newcommand\T{\mathbb T}
\newcommand\D{\mathbb D}
\newcommand\Z{\mathbb Z}
\newcommand\e{{\rm e}}
\newcommand\B{\mathbb B}

\newcommand\eps{\varepsilon}
\newcommand\ind{\mathds{1}}

\newcommand\converge{\mathop{\longrightarrow}\limits}
\newcommand\capa{{\rm Cap}\,}
\let\phi=\varphi
\let\hat = \widehat
\let\tilde=\widetilde

\title{\bf Pluricapacity and approximation numbers of composition operators}
\author{\it Daniel~Li,  Herv\'e~Queff\'elec, L.~Rodr{\'\i}guez-Piazza}

\date{\footnotesize \today}

\begin{document}

\maketitle

\noindent {\bf Abstract.} For suitable bounded hyperconvex sets $\Omega$ in $\C^N$, in particular the ball or the polydisk, we give estimates for the 
approximation numbers of composition operators $C_\phi \colon H^2 (\Omega) \to H^2 (\Omega)$ when $\phi (\Omega)$ is relatively compact in $\Omega$, 
involving the Monge-Amp\`ere capacity of $\phi (\Omega)$. 
\medskip

\medskip

\noindent \emph{MSC~2010 numbers} -- \emph{Primary}\/: 47B33 -- \emph{Secondary}\/: 30H10  -- 31B15 -- 32A35 -- 32U20 -- 41A25 -- 41A35 -- 46B28 
-- 46E20 -- 47B06

%%%%%%%%%%%%%%%%%%%%%%%%%%%%%%%%%%%%%%%%%%%%%%%%%%%%%%%%%%%%%%%%%%%%%%%%%%%%%

\section{Introduction} \label{sec: intro}

Let $\D$ be the unit disk in $\C$, $H^{2}(\D)$ the corresponding Hardy space,  $\varphi$  a non-constant analytic self-map  of $\D$ and 
$C_\varphi \colon H^{2}(\D) \to H^{2}(\D)$ the associated composition operator. In \cite{LQR}, we proved a formula connecting the approximation 
numbers $a_{n} (C_\varphi)$ of $C_\varphi$, and the Green capacity of the image $\varphi (\D) $ in $\D$, namely, when 
$\overline{[\phi^{\phantom{l}}\hskip - 3 pt (\D)]} \subset \D$, 
we have:
\begin{equation} \label{form}  
\beta(C_\varphi) := \lim_{n\to \infty} [a_{n}(C_\varphi) ]^{1/n} = \exp \big(- 1/\capa [\phi (\D)] \big) \, ,
\end{equation} 
where  $\capa [\phi (\D)]$ is the \emph{Green capacity} of $\phi (\D)$.\par
\smallskip

A non-trivial consequence of that formula was the following:
\begin{equation} \label{nt} 
\Vert \varphi \Vert_\infty = 1 \quad \Longrightarrow \quad  a_{n}(C_\varphi)\geq \delta \, \e^{- n \eps_n} \text{ where } \eps_n \to 0_{+} \, .
\end{equation} 
In other terms, as soon as $ \Vert \varphi\Vert_\infty = 1$, we cannot hope better for the numbers $a_{n} (C_\varphi)$ than a subexponential decay, 
however slowly $\varepsilon_n$ tends to $0$. 
\smallskip

In \cite{LIQR}, we pursued that line of investigation in dimension $N \geq 2$, namely on $H^{2}(\D^N)$, and showed that in some cases 
the implication \eqref{nt} still holds (\cite[Theorem~3.1]{LIQR}): 
\begin{equation} \label{Nt}
\Vert \varphi \Vert_\infty = 1 \quad \Longrightarrow \quad  a_{n}(C_\varphi)\geq \delta \, \e^{- n^{1/ N} \eps_n} \text{ where } \eps_n \to 0_{+} \, ,
\end{equation} 
(the substitution of $n$ by $n^{1/N}$ is mandatory as shown by the results of \cite{BLQR}). 
\smallskip

We show in this paper that, in general, for non-degenerate symbols, we have similar formulas  to \eqref{form} at our disposal for the parameters: 
\begin{equation} \label{para} 
\beta_{N}^{-}(C_\varphi) = \liminf_{n \to \infty} [ a_{n^N} (C_\varphi) ]^{1/n} \quad \text{and} \quad 
\beta_{N}^{+}(C_\varphi) = \limsup_{n \to \infty} [ a_{n^N} (C_\varphi) ]^{1/n} \, . 
\end{equation} 
These bounds are given in terms of the Monge-Amp{\`e}re (or Bedford-Taylor) capacity of $\varphi (\D^N)$ in $\D^N$, a notion which is the natural 
multidimensional extension of the Green capacity when the dimension $N$ is $\geq 2$ (\cite[Theorem~~6.4]{LIQR}). We show that we have 
$\beta_{N}^{-}(C_\varphi) =  \beta_{N}^{+}(C_\varphi)$ for well-behaved symbols.
\par\medskip

%%%%%%%%%%%%%%%%%%%%%%%%%%%%%%%%%%%%%%%%%%%%%%%%%%%%%%%%%%%%%%%%%%%%%%%%%%%%
\section{Notations and background} 

\subsection{Complex analysis}

Let $\Omega$ be a domain in $\C^N$; a function $u \colon \Omega \to \R \cup \{- \infty \}$ is said \emph{plurisubharmonic} ({\it psh}) if it is u.s.c. and if 
for every complex line $L = \{ a + z w \, ; \ z \in \C\}$ ($a \in \Omega$, $w \in \C^N$), the function $z \mapsto u ( a + z w)$ is subharmonic in 
$\Omega \cap L$. 
We denote ${\cal PSH} (\Omega)$ the set of plurisubharmonic  functions in $\Omega$. If $f \colon \Omega \to \C$ is holomorphic, then $\log |f|$ and 
$|f|^\alpha$, $\alpha > 0$, are {\it psh}. Every real-valued convex function is {\it psh} (convex functions are those whose composition with all $\R$-linear 
isomorphisms are subharmonic, though plurisubharmonic functions are those whose composition with all $\C$-linear isomorphisms are subharmonic: see 
\cite[Theorem~2.9.12]{Klimek}).

Let $dd^c = 2 i \partial \bar\partial$, and $(dd^c)^N = dd^c \wedge \cdots \wedge dd^c$ ($N$ times). When 
$u \in {\cal PSH} (\Omega) \cap\, {\cal C}^2 (\Omega)$, we have:
\begin{displaymath} 
(dd^c u)^N = 4^N N! \det \bigg( \frac{\partial^2 u}{\partial z_j \partial \bar{z}_k} \bigg)\, d\lambda_{2N} (z) \, , 
\end{displaymath} 
where $d \lambda_{2N} (z) = (i/2)^N dz_1 \wedge d\bar{z}_1 \wedge \cdots \wedge dz_N \wedge d\bar{z}_N$ is the usual volume in $\C^N$. 
In general, the current $(dd^c u)^N$ can be defined for all locally bounded $u \in {\cal PSH} (\Omega)$ and is actually a positive measure on $\Omega$ 
(\cite{BT}). 
\smallskip

Given $p_1, \ldots, p_J \in \Omega$, the pluricomplex Green function with poles $p_1, \ldots, p_J$ and weights $c_1, \ldots, c_N > 0$ is defined as:
\begin{align*} 
g (z)  = g (z, p_1, \ldots, p_J)  =  \sup \{ v (z) & \, ;  \ v \in {\cal PSH} (\Omega) \, ,  v \leq 0  \text{ and } \\
& v (z) \leq c_j \log |z - p_j| + {\rm O}\, (1) \, ,  \forall j = 1, \ldots, J \} \, .
\end{align*}
In particular, for $J = 1$ and $p_1 = a$, $c_1 = 1$, $g (\, \cdot \, , a)$ is the \emph{pluricomplex Green function} of $\Omega$ with pole $a \in \Omega$. If 
$0 \in \Omega$ and $a = 0$, we denote it by $g_\Omega$ and call it the \emph{pluricomplex Green function of $\Omega$}; hence:
\begin{displaymath} 
g_a (z) =  g (z, a) = \sup \{ u (z) \, ; \ u \in {\cal PSH} (\Omega)\, , \ u \leq 0 \text{ and } u (z) \leq \log | z - a | + {\rm O}\, (1) \} \, .
\end{displaymath} 

Let $\Omega$ be an open subset of $\C^N$. A continuous function $\rho \colon \Omega \to \R$ is an exhaustion function if 
there exists $a \in (- \infty, + \infty]$ such that $\rho (z) < a$ for all $z \in \Omega$, and the set 
$\Omega_c = \{ z \in \Omega \, ; \ \rho (z) < c \}$ is relatively compact in $\Omega$ for every $c < a$.

A domain $\Omega$ in $\C^N$ is said \emph{hyperconvex} if there exists a continuous {\it psh} exhaustion function $\rho \colon \Omega \to  (-\infty, 0)$ 
(see \cite[p.~80]{Klimek}). We may of course replace the upper bound $0$ by any other real number. Without this upper bound, $\Omega$ is said 
\emph{pseudoconvex}.
\par\smallskip

Let $\Omega$ be a hyperconvex domain, with negative continuous {\it psh} exhaustion function $\rho$ and $\mu_{\rho, r}$ the associated 
Demailly-Monge-Amp\`ere measures,  defined as:
\begin{equation} \label{mesures DMA}
\mu_{u, r} = (dd^c u_r)^N - \ind_{\Omega \setminus B_{\Omega, u} (r)} (dd^c u)^N \, , 
\end{equation} 
for $r < 0$, where $u_r = \max (u, r)$ and:
\begin{displaymath} 
B_{\Omega, u} (r) = \{ z \in \Omega \, ; \ u (z) < r \} \, . 
\end{displaymath} 
The nonnegative measure $\mu_{u, r}$ is supported by $S_{\Omega, u} (r) := \{ z \in \Omega \, ; \ u (z) = r \}$. 

If:
\begin{displaymath} 
\int_\Omega (dd^c \rho)^N < \infty \, ,
\end{displaymath} 
these measures, considered as measures on $\overline\Omega$, weak-$\ast$ converge, as $r$ goes to $0$, to a positive measure 
$\mu = \mu_{\Omega, \rho}$ supported by $\partial \Omega$ and with total mass $\int_\Omega (dd^c \rho)^N$ (\cite[Th\'eor\`eme~3.1]{Demailly-1987}, 
or \cite[Lemma~6.5.10]{Klimek}). 
\par\smallskip

For the pluricomplex Green function $g_a$ with pole $a$, we have $(dd^c g_a)^N = (2 \pi)^N \delta_a$ (\cite[Th\'eor\`eme~4.3]{Demailly-1987}) and 
$g_a (a) = - \infty$, so $a \in B_{\Omega, g_a} (r)$ for every \hbox{$r < 0$} and $\ind_{\Omega \setminus B_{\Omega, g_a} (r)} (dd^c g_a)^N = 0$. Hence 
the Demailly-Monge-Amp\`ere mesure $\mu_{g_a, r}$ is equal to $\big(dd^c (g_a)_r \big)^N$. By \cite[Lemma~1]{Tiba}, we have  
$(1 / |r| ) \,\big(dd^c (g_a)_r \big)^N =  u_{\bar{B}_{\Omega, g_a} (r), \Omega}$,  the relative extremal function of 
$\bar{B}_{\Omega, g_a} (r) =\{z \in \Omega \, ; \ g_a (z) \leq r\}$ in $\Omega$ (see \eqref{fct extremale} for the definition), and this measure is supported, 
not only by $S_{\Omega, g_a} (r)$, but merely by the Shilov boundary of $\bar{B}_{\Omega, g_a} (r)$ (see Section~\ref{sec: Hardy on bsd} for the definition). 

Since $(dd^c g_a)^N  = (2 \pi)^N \delta_a$ has mass $(2 \pi)^N < \infty$, these measures weak-$\ast$ converge, as $r$ goes to $0$, to a positive measure 
$\mu = \mu_{\Omega, g_a}$ supported by $\partial \Omega$ with mass $(2 \pi)^N$. Demailly (\cite[D\'efinition~5.2]{Demailly-1987} call the measure 
$\frac{1}{(2 \pi)^N} \, \mu_{\Omega, g_a}$ the \emph{pluriharmonic measure of $a$}. When $\Omega$ is balanced ($a z \in \Omega$ for every 
$z \in \Omega$ and $|a| = 1$), the support of this pluriharmonic measure is the Shilov boundary of $\overline\Omega$ (\cite[very end of the paper]{Tiba}).
\par\medskip

A \emph{bounded symmetric domain} of $\C^N$ is a bounded open and convex subset $\Omega$ of $\C^N$ which is circled ($a z \in \Omega$ for 
$z \in \Omega$ and $|a| \leq 1$) and such that for every point $a \in \Omega$, there is an involutive bi-holomorphic map 
$\gamma \colon \Omega \to \Omega$ such that $a$ is an isolated fixed point of $\gamma$ (equivalently, $\gamma (a) = a$ and $\gamma ' (a) = - id$: see 
\cite[Proposition 3.1.1]{VIGUE}). For this definition, see \cite[Definition~16 and Theorem~17]{CLERC}, or \cite[Definition~5 and Theorem~4]{CLERC2}. 
Note that the convexity is automatic (Hermann Convexity Theorem; see \cite[p.~503 and Corollary~4.10]{Kaup}). \'E.~Cartan showed that every bounded 
symmetric domain of $\C^N$ is homogeneous, i.e. the group $\Gamma$ of automorphisms of $\Omega$ acts transitively on $\Omega$: for every 
$a, b \in \Omega$, there is an automorphism $\gamma$ of $\Omega$ such that $\gamma (a) = b$ (see \cite[p.~250]{VIGUE}). Conversely, every homogeneous 
bounded convex domain is symmetric, since $\sigma (z) = - z$ is a symmetry about $0$  (see \cite[p.~250]{VIGUE} or \cite[Remark~2.1.2~(e)]{Jarnicki}). 
Moreover, each automorphism extends continuously to $\overline \Omega$ (see \cite{HAHN-MITCHELL-TAMS}). 

The unit ball $\B_N$ and the polydisk $\D^N$ are examples of bounded symmetric domains. Another example is, for $N = p \, q$, bi-holomorphic to the open 
unit ball of ${\cal M} (p, q) = {\cal L} (\C^q, \C^p)$ for the operator norm (see \cite[Theorem~4.9]{Kaup}). Every product of bounded symmetric domains is 
still a bounded symmetric domain. In particular, every product of balls $\Omega = \B_{l_1} \times \cdots \times \B_{l_m}$, $l_1 + \cdots + l_m = N$, is a 
bounded symmetric domain. 
\par

If $\Omega$ is a bounded symmetric domain, its gauge is a norm $\| \, . \, \|$ on $\C^N$ whose open unit ball is $\Omega$. Hence every bounded symmetric 
domain is hyperconvex (take $\rho (z) = \| z\| - 1$).

%%%%%%%%%%%%%%%%%%%%%%%%%%%%%%%%%%%%%%%%%%%%%%%%
\subsection{Hardy spaces on hyperconvex domains}

\subsubsection{Hardy spaces on bounded symmetric domains} \label{sec: Hardy on bsd}

We begin by defining the Hardy space on a bounded symmetric domain, because this is easier.

The \emph{Shilov boundary} (also called the Bergman-Shilov boundary or the distinguished boundary) $\partial_S \Omega$ of a bounded domain 
$\Omega$ is the smallest closed set $F \subseteq \partial \Omega$ such that $\sup_{z \in \overline{\Omega}} | f (z) | = \sup_{z \in F} | f (z) |$ for every 
function $f$ holomorphic in some neighborhood of $\overline \Omega$ (see \cite[\S~4.1]{CLERC}). 

When $\Omega$ is a bounded symmetric domain, it is also, since $\Omega$ is convex, the Shilov boundary of the algebra $A (\Omega)$ of the continuous 
functions on $\overline \Omega$ which are holomorphic in $\Omega$ (because every function $f \in A (\Omega)$ can be approximated by 
$f_\eps$ with $f_\eps (z) = f \big( \eps z_0 + (1 - \eps) z \big)$, where $z_0 \in \Omega$ is given: see \cite[pp.~152--154]{GUI}). 

The Shilov boundary of the ball $\B_N$ is equal to its topological boundary, but the Shilov boundary of the bidisk is 
$\partial_S {\D^2} = \{ (z_1, z_2) \in \C^2 \, ; \ |z_1| = |z_2| = 1\}$, whereas, its usual boundary $\partial \D^2$ is 
%$\{ (z_1, z_2) \in \C^2 \, ; \ |z_1|, |z_2| \leq 1 \text{ and } |z_1| = 1 \text{ or } |z_2| = 1 \}$
$(\T \times \overline \D) \cup (\overline \D \times \T)$; for the unit ball $B_N$, the Shilov 
boundary is equal to the usual boundary ${\mathbb S}^{N - 1}$ (\cite[\S~4.1]{CLERC}). Another example of a bounded symmetric domain, in $\C^3$, is 
the set $\Omega = \{ (z_1, z_2, z_3) \in \C^3 \, ; \ |z_1|^2 + |z_2|^2 < 1 \, , \ |z_3| < 1\}$ and its Shilov boundary is 
$\partial_S \Omega = \{ (z_1, z_2, z_3) \, ; \ |z_1|^2 + |z_2|^2 = 1 \, , \ |z_3| = 1\}$. For $p \geq q$, the matrix $A$ is in the topological boundary of 
${\cal M} (p, q)$ if and only if $\| A \| = 1$, but $A$ is in the Shilov boundary if and only if $A^\ast A = I_q$; therefore the two boundaries 
coincide if and only if $q = 1$, i.e. $\Omega = \B_N$ (see \cite[Example~2, p.~30]{CLERC2}).

Equivalently (see \cite[Corollary~9]{Harris}, or \cite[Theorem~33]{CLERC}, \cite[Theorem~10]{CLERC2}), $\partial_S \Omega$ is the set of the extreme 
points of the convex set $\overline \Omega$. \par

The Shilov boundary $\partial_S \Omega$ is invariant by the group $\Gamma$ of automorphisms of $\Omega$ and the subgroup 
$\Gamma_0 = \{ \gamma \in \Gamma \, ; \ \gamma (0) = 0 \}$ act transitively on $\partial_S \Omega$ (see \cite{HAHN-MITCHELL-TAMS}). A theorem 
of H.~Cartan states that the elements of $\Gamma_0$ are linear transformations of $\C^N$ and commute with the rotations (see \cite[Theorem~1]{Harris} or 
\cite[Proposition~2.1.8]{Jarnicki}). 
It follows that the Shilov boundary of a bounded symmetric domain $\Omega$ coincides with its topological boundary only for $\Omega = \B_N$ 
(see \cite[p.~572]{Li-Queff} or \cite[p.~367]{Li-Queff-Cambridge-1}); in particular the open unit ball of $\C^N$ for the norm $\| \, . \, \|_p$, 
$ 1 < p < \infty$, is never a bounded symmetric domain, unless $p = 2$.

The unique $\Gamma_0$-invariant probability measure $\sigma$ on $\partial_S \Omega$ is the normalized surface area (see \cite{HAHN-MITCHELL-TAMS}).
Then the \emph{Hardy space} $H^2 (\Omega)$ is the space of all complex-valued holomorphic functions $f$ on $\Omega$ such that:
\begin{displaymath} 
\| f \|_{H^2 (\Omega)} := \bigg( \sup_{0 < r < 1} \int_{\partial_S \Omega} | f (r \xi )|^2 \, d \sigma (\xi) \bigg)^{1/2} 
\end{displaymath} 
is finite (see \cite{HAHN-MITCHELL-TAMS} and \cite{HAHN-MITCHELL}). It is known that the integrals in this formula are non-decreasing as $r$ increases 
to $1$, so we can replace the supremum by a limit. The same definition can be given when $\Omega$ is a bounded complete Reinhardt domain (see 
\cite{Aizenberg}).

The space $H^2 (\Omega)$ is a Hilbert space (see \cite[Theorem~5]{HAHN-MITCHELL-TAMS}) and for every $z \in \Omega$, the evaluation map 
$f \in H^2 (\Omega) \mapsto f (z)$ is uniformly bounded on compacts subsets of $\Omega$, by a depending only on that compact set, and of $\Omega$ 
(\cite[Lemma~3]{HAHN-MITCHELL-TAMS}). 

For every $f \in H^2 (\Omega)$, there exists a boundary values function $f^\ast$ such that 
$\| f_r - f^\ast \|_{L^2 (\partial_S \Omega)} \converge_{r \to 1} 0$, where $f_r (z) = f (r z)$ (\cite[Theorem~3]{Bochner}), and the map 
$f \in H^2 (\Omega) \mapsto f^\ast \in L^2 (\partial_S \Omega)$ is an isometric embedding (\cite[Theorem~6]{HAHN-MITCHELL-TAMS}).

%%%%%%%%%%%%%%%%%%%%%%%%%%%%%%
\subsubsection{Hardy spaces on hyperconvex domains} 

For hyperconvex domains, the definition of Hardy spaces is more involved. It was done by E.~Poletsky and M.~Stessin 
(\cite[Theorem~6]{Pol-Stess}). Those domains are associated to a continuous negative {\it psh} exhaustion function $u$ on $\Omega$ and the definition of the 
Hardy spaces uses the Demailly-Monge-Amp\`ere measures. The space $H^2_u (\Omega)$ is the space of all holomorphic functions $f \colon \Omega \to \C$ 
such that:
\begin{displaymath} 
\sup_{r < 0} \int_{S_{\Omega, u} (\Omega)} |f|^2 \, d \mu_{u, r}  < \infty
\end{displaymath} 
and its norm is defined by:
\begin{displaymath} 
\| f \|_{H^2_u (\Omega)} =  \sup_{r < 0} \bigg( \frac{1}{(2 \pi)^N} \int_{S_{\Omega, u} (\Omega)} |f|^2 \, d \mu_{u, r} \bigg)^{1/2} \, .
\end{displaymath} 
We can replace the supremum by a limit since the integrals are non-decreasing as $r$ increases to $0$ (\cite[Corollaire~1.9]{Demailly-1987}.

The space $H^\infty (\Omega)$ of bounded holomorphic functions in $\Omega$ is contained in $H^2_u (\Omega)$ (see \cite{Pol-Stess}, remark before 
Lemma~3.4).

These spaces $H^2_u (\Omega)$ are Hilbert spaces (\cite[Theorem~4.1]{Pol-Stess}), but depends on the exhaustion function $u$ (even when $N = 1$: see for 
instance \cite{Sahin}). Nevertheless, they all coincide, with equivalent norms, for the functions $u$ for which the measure $(dd^c u)^N$ is compactly supported 
(\cite[Lemma~3.4]{Pol-Stess}); this is the case when $u (z) = g (z, a)$ is the pluricomplex Green function with pole $a \in \Omega$ (because then 
$(dd^c u)^N = (2 \pi)^N \delta_a$: see \cite[Th\'eor\`eme~4.3]{Demailly-1987}, or \cite[Theorem~6.3.6]{Klimek}). 

When $\Omega$ is the ball $\B_N$ and $u (z) = \log\| z \|_2$, then $(dd^c u)^N = C\, \delta_0$ and $\mu_{u, r} = (2 \pi)^N d \sigma_t$, where 
$d \sigma_t$ is the normalized surface area on the sphere of radius $t := \e^r$ (see \cite[Section 4]{Pol-Stess} or \cite[Example~3.3]{Demailly}).  
When $\Omega$ is the polydisk $\D^N$ and $u (z) = \log \| z \|_\infty$, then $(dd^c u)^N = (2 \pi)^N \delta_0$ (\cite[Corollary~5.4]{Demailly-1991}) 
and $\frac{1}{(2 \pi)^N} \,\mu_{u, r}$ is the Lebesgue measure of the torus $r\T^N$ (see \cite[Example~3.10]{Demailly}). Note that in \cite{Demailly} and 
\cite{Demailly-1991}, the operator $d^c$ is defined as $\frac{i}{2 \pi} (\bar\partial - \partial)$ instead of $i (\bar\partial - \partial)$, as usually used.

In these two cases, the Hardy spaces are the same as the usual ones (see \cite[Remark~5.2.1]{Alan}). 
\smallskip

In the sequel, we only consider the exhausting function $u = g_\Omega$; hence we will write $B_\Omega (r)$, $S_\Omega (r)$ and $H^2 (\Omega)$ instead of 
$B_{\Omega, u} (r)$, $S_{\Omega, u} (r)$ and $H^2_u (\Omega)$.
\smallskip

The two notions of Hardy spaces for a bounded symmetric domain are the same:
\begin{proposition} \label{Hardy spaces}
Let $\Omega$ be a bounded symmetric domain in $\C^N$. Then the Hardy space $H^2 (\Omega)$ coincides with the subspace of the Poletski-Stessin Hardy 
space $H_{g_\Omega} (\Omega)$, with equality of the norms.
\end{proposition}
\begin{proof} First let us note that if $\| \, . \, \|$ is the norm whose open unit ball is $\Omega$, then $g_\Omega (z) = \log \| z \|$ (see 
\cite[Proposition~3.3.2]{Blocki-Jag}).\par

Let $\mu_\Omega$ be the measure which is the $\ast$-weak limit of the Demailly-Monge-Amp\`ere measures $\mu_r = \big(dd^c (g_\Omega)_r \big)^N$. We 
saw that it is supported by $\partial_S \Omega$. By the remark made in \cite[pp.~536-537]{Demailly-1987}, since the automorphisms of $\Omega$ 
continuously extend on $\partial \Omega$, the measure $\mu_\Omega$ is $\Gamma$-invariant. By unicity, the harmonic measure 
$\tilde \mu_\Omega = (2 \pi)^{- N} \mu_\Omega$ at $0$ hence coincides with the normalized area measure on $\partial_S \Omega$. We have, for 
$f \colon \Omega \to \C$ holomorphic and $0 < s < 1$:
\begin{displaymath}
\int_{\partial_S \Omega} | f (s z) |^2 \, d\tilde\mu_\Omega (z) = \int_{\partial \Omega} | f (s z) |^2 \, d\tilde\mu_\Omega (z) 
= \lim_{r \to 0} \frac{1}{(2 \pi)^N} \int_{S_{\Omega} (r)} | f (s z) |^2 \, d \mu_r (z) \, ,
\end{displaymath}
because $z \mapsto |f (sz)|^2$ is continuous on $\overline \Omega$. Now, since $g_\Omega (z) = \log \| z\|$, we have 
$S_\Omega (r) = \e^r \partial \Omega$ and $(g_\Omega)_r (z) + t = (g_\Omega)_{r + t} (s z)$; hence $\mu_r (s A) = \mu_{r + t} (A)$ for every Borel subset 
$A$ of $\partial \Omega$, where $t = \log s$. It follows that:
\begin{displaymath} 
\int_{S_{\Omega} (r)} | f (s z) |^2 \, d \mu_r (z) = \int_{S_\Omega (r + t)} |f (\zeta)|^2 \, d\mu_{r + t} (\zeta) \, . 
\end{displaymath} 
By letting $r$ and $t$ going to $0$, we get:
\begin{displaymath} 
\| f \|_{H^2 (\Omega)}^2 = \lim_{r, t \to 0} \frac{1}{(2 \pi)^N} \int_{S_\Omega (r + t)} |f (\zeta)|^2 \, d\mu_{r + t} (\zeta) = 
\| f \|_{H^2_{g_\Omega}}^2 \, ;
\end{displaymath} 
hence $f \in H^2 (\Omega)$ if and only if $f \in H^2_{g_0} (\Omega)$, with the same norms.
\end{proof}

We have (\cite[Theorem~3.6]{Pol-Stess}): 
\begin{proposition} [Poletsky-Stessin] \label{eval}
For every $z \in \Omega$, the evaluation map $f \in H^2 (\Omega) \mapsto f (z)$ is uniformly bounded on compacts subsets of $\Omega$, by a constant 
depending only on that compact set, and of $\Omega$.
\end{proposition}

Hence $H^2 (\Omega)$ has a \emph{reproducing kernel}, defined by:
\begin{equation} 
f (a) = \langle f, K_a \rangle \, , \quad \text{for } f \in H^2 (\Omega) \, ,
\end{equation}
and for each $r < 0$:
\begin{equation}\label{sup} 
L_r := \sup_{a \in \overline{B_\Omega (r)}} \Vert K_{a}\Vert_2 < \infty \, .
\end{equation}
%

%%%%%%%%%%%%%%%%%%%%%%%%%%%%%%%%
\subsection{Composition operators}

A Schur map, associated with the bounded hyperconvex domain $\Omega$, is a \emph{non-constant} analytic map of $\Omega$ into itself. It is said 
\emph{non degenerate} if its Jacobian is not identically null. It is equivalent to say that the differential $\phi ' (a) \colon \C^N \to \C^N$ is an invertible linear 
map for at least one point $a \in \Omega$. In \cite{BLQR}, we used the terminology \emph{truly $N$-dimensional}. Then, by the implicit function theorem, this 
is equivalent to saying that $\phi (\Omega)$ has non-void interior. We say that the Schur map $\phi$ is a \emph{symbol} if it defines a \emph{bounded} 
composition operator $C_\phi \colon H^2 (\Omega) \to H^2 (\Omega)$ by $C_\phi (f) = f \circ \phi$. \par

Let us recall that although any Schur function generates a bounded composition operator on $H^2 (\D)$, this is no longer the case on $H^2 (\D^N)$ as soon as 
$N \geq 2$, as shown for example by the Schur map $\varphi (z_1, z_2) = (z_1, z_1)$.  Indeed (see \cite{BAYPOLY}), if say $N = 2$, taking 
$f (z) = \sum_{j = 0}^n z_1^j z_2^{n - j}$, we see that:
\begin{displaymath} 
\Vert f \Vert_{2} = \sqrt{n + 1} \quad \text{while} \quad \Vert C_{\phi}f \Vert_{2} = \Vert (n + 1) z_1^n \Vert_2 = n + 1 \, .
\end{displaymath} 
The same phenomenon occurs on $H^2 (\B_N)$ (\cite{MCCL}; see also \cite{CSW}, \cite{CW}, and \cite{CO-MC}; see also \cite{Pol-Stess}).

%%%%%%%%%%%%%%%%

\subsection{$s$-numbers of operators on a Hilbert space}
 
We begin by recalling a few operator-theoretic facts. Let $H$ be a Hilbert space. The approximation numbers $a_{n} (T) = a_n$ of an operator 
$T \colon H\to H$ are defined as:
\begin{equation} 
\qquad \quad a_n = \inf_{{\rm rank}\, R < n} \Vert T - R\Vert \, , \qquad  n = 1, 2, \ldots 
\end{equation} 
The operator $T$ is compact if and only if $\lim_{n\to \infty} a_{n}(T) = 0$. 

According to a result of Allahverdiev \cite[p.~155]{CAST}, $a_n = s_n$, the $n$-th singular number of $T$, i.e. the $n$-th eigenvalue of 
$|T| := \sqrt{T^{\ast}T}$ when those eigenvalues are rearranged in non-increasing order. 

The $n$-th width $d_{n}(K)$ of a subset $K$ of a Banach space $Y$ measures the defect of flatness of $K$ and is by definition:
\begin{equation} 
d_{n}(K) = \inf_{\dim E < n} \bigg[ \sup_{f \in K} {\rm dist}\, (f, E) \bigg] \, ,
\end{equation} 
where $E$ runs over all subspaces of $Y$ with dimension $< n$ and where ${\rm dist}\, (f, E)$ denotes the distance of $f$ to $E$. If $T \colon X \to Y$ is an 
operator between Banach spaces, the $n$-th Kolmogorov number $d_{n} (T)$ of $T$  is the $n$th-width in $Y$ of $T (B_X)$ where $B_X$ is the closed unit ball 
of $X$, namely:
\begin{equation} 
d_{n} (T) = \inf_{\dim E < n} \bigg[ \sup_{f \in B_X} {\rm dist}\, (Tf, E) \bigg] \, .
\end{equation} 
In the case where $X = Y = H$, a Hilbert space, we have:
\begin{equation} \label{a=d}
 a_{n} (T) = d_{n} (T) \quad \text{for all } n \geq 1 \, ,
\end{equation}
and (\cite{LQR}) the following alternative definition of $a_{n}(T)$: 
\begin{equation} \label{alter} 
a_{n} (T) = \inf_{\dim E < n} \bigg[ \sup_{f \in B_H} {\rm dist}\, (Tf, TE) \bigg] \,.
\end{equation}

In this work, we use, for an operator $T \colon H \to H$, the following notation:
\begin{equation} \label{liminf}
\beta_N^- (T) = \liminf_{n \to \infty} [a_{n^N} (T)]^{1 / n}
\end{equation} 
and:
\begin{equation} \label{limsup}
\beta_N^+ (T) = \limsup_{n \to \infty} [a_{n^N} (T)]^{1 / n} \, .
\end{equation} 
When these two quantities are equal, we write them $\beta_N (T)$.

%%%%%%%%%%%%%%%%%%%%%%%%%%%%%%%%%%%%%%%%%%%%%%%%%%%%%%%%%%%%%%%%%%%%%%%%%%%%%
\section{Pluripotential theory}

\subsection{Monge-Amp{\`e}re capacity}

Let $K$ be a compact subset of an open subset $\Omega$ of $\C^N$. The \emph{Monge-Amp\`ere capacity} of $K$ has been defined by Bedford and Taylor 
(\cite{BT}; see also \cite[Part~II, Chapter~1]{Klimek}) as:
\begin{displaymath} 
\capa (K) = \sup \bigg\{ \int_K (dd^c u)^N \, ; \ u \in {\cal PSH} (\Omega) \text{ and } 0 \leq u \leq 1 \text{ on } \Omega \bigg\} \, .
\end{displaymath} 

When $\Omega$ is bounded and hyperconvex, we have a more convenient formula (\cite[Proposition~5.3]{BT}, \cite[Proposition~4.6.1]{Klimek}):
\begin{equation} \label{more convenient}
\capa (K) = \int_{\Omega} (dd^c u_K^\ast)^N = \int_K (dd^c u_K^\ast)^N \, ,
\end{equation} 
(the positive measure $(dd^c u_K^\ast)^N$ is supported by $K$; actually by $\partial K$: see \cite[Properties~8.1~(c)]{Demailly}), where 
$u_K = u_{K, \Omega}$ is the \emph{relative extremal function} of $K$, defined, for  any subset $E \subseteq \Omega$, as:
\begin{equation} \label{fct extremale}
u_{E, \Omega} = \sup\{ v \in {\cal PSH} (\Omega) \, ; \ v \leq 0 \text{ and } v \leq - 1 \text{ on } E\} \, ,
\end{equation} 
and $u_{E, \Omega}^\ast$ is its upper semi-continuous regularization: 
\begin{displaymath} 
\qquad \quad u_{E, \Omega}^\ast (z) = \limsup_{\zeta \to z} u_{E, \Omega} (\zeta) \, , \qquad z \in \Omega \, ,
\end{displaymath} 
called the \emph{regularized relative extremal function} of $E$.
\smallskip\goodbreak

For an open subset $\omega$ of $\Omega$, its capacity is defined as:
\begin{displaymath} 
\capa (\omega) = \sup \{ \capa (K) \, ; \ K \text{ is a compact subset of } \omega \} \, .
\end{displaymath} 
When $\overline{\omega} \subset \Omega$ is a compact subset of $\Omega$, we have (\cite[equation~$(6.2)$]{BT}, \cite[Corollary~4.6.2]{Klimek}):
\begin{equation} \label{capa ouvert}
\capa (\omega) = \int_{\Omega} (dd^c u_\omega)^N \, .
\end{equation} 
The \emph{outer capacity} of a subset $E \subseteq \Omega$ is:
\begin{displaymath} 
\capa^\ast (E) = \inf \{ \capa (\omega) \, ; \ \omega \supseteq E \text{ and } \omega \text{ open} \} \, .
\end{displaymath} 
If $\Omega$ is hyperconvex and $E$ relatively compact in $\Omega$, then (\cite[Proposition~4.7.2]{Klimek}):
\begin{displaymath} 
\capa^\ast (E) = \int_\Omega (dd^c u_{E, \Omega}^\ast)^N \, .
\end{displaymath} 
\goodbreak

\noindent{\bf Remark.} A.~Zeriahi (\cite{ZER}) pointed out to us the following result.
\begin{proposition} 
Let $K$ be a compact subset of $\Omega$. Then:
\begin{displaymath} 
\capa (K) = \capa (\partial K) \, .
\end{displaymath} 
\end{proposition}
\begin{proof}
Of course $u_K \leq u_{\partial K}$ since $\partial K \subseteq K$. Conversely, let $v \in {\cal PSH} (\Omega)$ non-positive 
such that $v \leq - 1$ on $\partial K$. By the maximum principle (see \cite[Corollary~2.9.6]{Klimek}), we get that $v \leq - 1$ on $K$. Hence $v \leq u_K$. 
Taking the supremum over all those $v$, we obtain $u_{\partial K} \leq u_K$, and therefore $u_{\partial K} = u_K$.

By \eqref{more convenient}, it follows that:
\begin{equation} 
\capa (K) = \int_{\Omega} (dd^c u_K^\ast)^N =  \int_{\Omega} (dd^c u_{\partial K}^\ast)^N = \capa (\partial K) \, . \qedhere
\end{equation} 
\end{proof}
%

%%%%%%%%%%%%%%%%%%%%%%%%%%%%%%%%%%%%%%%%%%%%%%%%%%%%
\subsection{Regular sets}

Let $E \subseteq \C^N$ be bounded. Recall that the polynomial convex hull of $E$ is:
\begin{displaymath}
\hat E = \{ z \in \C \, ; \ | P (z) | \leq \sup_E |P| \text{ for every polynomial } P \} \, .
\end{displaymath}

A point $a \in \hat E$ is called \emph{regular} if $u^\ast_{E, \Omega} (a) = - 1$ 
for an open set $\Omega \supseteq \hat E$ (note that we always have $u_{E, \Omega} = u_{E, \Omega}  = - 1$ on the interior of $E$: see 
\cite[Properties~8.1~(c)]{Demailly}). The set $E$ is said to be \emph{regular} if all points of $\hat E$ are regular.

The \emph{pluricomplex Green function} of $E$, also called the \emph{$L$-extremal function} of $E$, is defined, for $z \in \C^N$, as:
\begin{displaymath} 
V_E (z) = \sup \{ v (z) \, ; \ v \in {\cal L}\, , \quad v \leq 0 \text{ on }  E \} \, ,
\end{displaymath} 
where ${\cal L}$ is the \emph{Lelong class} of all functions $v \in {\cal PSH} (\C^N)$ such that, for some constant $C > 0$:
\begin{displaymath} 
\qquad v (z) \leq C + \log (1 + |z|) \quad \text{for all } z \in \C^N \, .
\end{displaymath} 
A point $a \in \hat E$ is called \emph{$L$-regular} if $V_E^\ast (a) = 0$, where $V_E^\ast$ is the upper semicontinuous regularization of $V_E$.  The set 
$E$ is \emph{$L$-regular} if all points of $\hat E$ are $L$-regular. 

By \cite[Proposition~2.2]{Klimek-art} (see also \cite[Proposition~5.3.3, and Corollary~5.3.4]{Klimek}), for $E$ bounded and non pluripolar, and $\Omega$ a 
bounded open neigbourhood of $\hat E$, we have:
\begin{equation} \label{equiv u-V}
m (u_{E, \Omega} + 1) \leq V_E \leq M (u_{E, \Omega} + 1) 
\end{equation} 
for some positive constants $m, M$. Hence the regularity of $a \in \hat E$ is equivalent to its $L$-regularity. 

Recall that $E$ is pluripolar if there exists an open set $\Omega$ containing $E$ and $v \in {\cal PSH} (\Omega)$ such that $E \subseteq \{ v = - \infty\}$. 
This is equivalent to say that there exists a hyperconvex domain $\Omega$ of $\C^N$ containing $E$ such that $u_{E, \Omega}^\ast \equiv 0$ (see 
\cite[Corollary~4.7.3 and Theorem~4.7.5]{Klimek}). By Josefson's theorem (\cite[Theorem~4.7.4]{Klimek}), $E$ is pluripolar if and only if there exists 
$v \in {\cal PSH} (\C^N)$ such that $E \subseteq \{ v = - \infty\}$. Recall also that $E$ is pluripolar if and only if its outer capacity $\capa^\ast (E)$ is null 
(\cite[Theorem~4.7.5]{Klimek}).

When $\Omega$ is hyperconvex and $E$ is compact, non pluripolar, the regularity of $E$ implies that $u_{E, \Omega}$ 
and $V_E$ are continuous, on $\Omega$ and $\C^N$ respectively (\cite[Proposition~4.5.3 and Corollary~5.1.4]{Klimek}). Conversely, if $u_{E, \Omega}$ 
is continuous, for some hyperconvex neighbourhood $\Omega$ of $E$, then $u_{E, \Omega} (z) = - 1$ for all $z \in E$; hence $V_E (z) = 0$ for all $z \in E$, 
by \eqref{equiv u-V}; but $V_E = V_{\hat E}$ when $E$ is compact (\cite[Theorem~5.1.7]{Klimek}), so $V_E (z) = 0$ for all $z \in \hat E$; by 
\eqref{equiv u-V} again, we obtain that $u_{E, \Omega} (z) = - 1$ for all $z \in \hat E$; therefore $E$ is regular. In the 
same way, the continuity of $V_E$ implies the regularity of $E$. These results are due to Siciak (\cite[Proposition~6.1 and Proposition~6.2]{Siciak}). 
\smallskip

Every closed ball $B = B (a, r)$ of an arbitrary norm $\| \, . \, \|$ on $\C^N$ is regular since its $L$-extremal function is:
\begin{displaymath} 
V_B (z) = \log^+ \big( \| z - a \| / r) 
\end{displaymath} 
(\cite[p.~179, \S~2.6]{Siciak}).

%%%%%%%%%%%%%%%%%%%%%%%%%%%%%%%%%%%%%%%%%%%%%%%%%%%%%%
\subsection{Zakharyuta's formula} 

We will need a formula that Zakharyuta, in order to solve a problem raised by Kolmogorov, proved, conditionally to a conjecture, called Zakharyuta's conjecture, 
on the uniform approximation of the relative  extremal function $u_{K, \Omega}$ by pluricomplex Green functions. This conjecture has been proved by Nivoche 
(\cite[Theorem~A]{Nivoche}), in a more general setting that we state below:
\begin{theorem} [Nivoche]
Let $K$ be a regular compact subset of a bounded hyperconvex domain $\Omega$ of $\C^N$. Then for every $\eps > 0$ and $\delta$ small enough, there exists 
a pluricomplex Green function $g$ on $\Omega$ with a finite number of logarithmic poles such that: \par
\smallskip
$1)$ the poles of $g$ lie in $W = \{ z \in \Omega \, ; \ u_K (z) < - 1 + \delta\}$; \par
\smallskip
$2)$ we have, for every $z \in \overline{\Omega} \setminus W$:
\begin{displaymath} 
(1 + \eps) \, g (z) \leq u_K (z) \leq (1 - \eps)\, g (z) \, .
\end{displaymath} 
\end{theorem} 

In order to state Zakharyuta's formula, we need some additional notations.

Let $K$ be a compact subset of $\Omega$ with non-empty interior, and $A_{K}$ the set of restrictions to $K$ of those functions that are analytic and bounded 
by $1$, i.e. those functions belonging to the unit ball $B_{H^{\infty} (\Omega)}$ of the space $H^{\infty}(\Omega)$ of the bounded analytic functions in 
$\Omega$, considered as a subset of the space $\mathcal{C}(K)$ of complex functions defined on $K$, equipped with the sup-norm on $K$. 

Let $d_{n} (A_{K}) $ be the $n$th-width of  $A_{K}$ in $\mathcal{C}(K)$, namely:
\begin{equation} 
d_{n} (A_{K}) = \inf_{L} \bigg[ \sup_{f \in A_{K}} {\rm dist}\, (f, L) \bigg] \, ,
\end{equation} 
where $L$ runs over all $k$-dimensional subspaces of  $\mathcal{C}(K)$, with $k < n$. 

Equivalently, $d_{n}(A_{K})$ is the $n$th-Kolmogorov number of the natural injection $J$ of  $H^{\infty}(\Omega)$ into  $\mathcal{C}(K)$ (recall that 
$K$ has non-empty interior). It is convenient to set, as in \cite{ZAK}:
\begin{equation} 
\tau_N (K) = \frac{1\ }{(2\pi)^N}\, \capa (K) 
\end{equation} 
and:
\begin{equation} 
\Gamma_N (K) = \exp \Bigg[ - \bigg( \frac{N!}{\tau_N (K)} \bigg)^{1/N} \Bigg] \, \raise 1pt \hbox{,}
\end{equation} 
i.e.:
\begin{equation} 
\Gamma_N (K) = \exp \bigg[ - 2 \pi \, \bigg(\frac{N!}{\capa (K)}\bigg)^{1/N} \bigg]\, . 
\end{equation} 
Observe that $\capa (K) > 0$ since we assumed that $K$ has non-empty interior. Now, we have (\cite[Theorem~5.6]{ZAK}; see also 
\cite[Theorem~5]{ZAK-2009} or \cite[pages 30--32]{Yazici}, for a detailed proof):
\begin{theorem} [Zakharyuta-Nivoche] \label{zak-niv}
Let $\Omega$ be a bounded hyperconvex domain and $K$ a regular compact subset of $\Omega$ with non-empty interior, which is holomorphically convex 
in $\Omega$ (i.e. $K = \tilde K_\Omega$). 
Then:
\begin{equation} 
- \log d_{n} (A_{K}) \sim \bigg( \frac{N!}{\tau_N (K)} \bigg)^{1/N} \, n^{1/N} \, .
\end{equation} 
\end{theorem} 

Here $\tilde K_\Omega$ is the holomorphic convex hull of $K$ in $\Omega$, that is:
\begin{displaymath}
\tilde K_\Omega = \{  z \in \Omega \, ; \ | f (z) | \leq \sup_K |f| \text{ for every } f \in {\cal O} (\Omega) \} \, ,
\end{displaymath}
where ${\cal O} (\Omega)$ is the set of all functions holomorphic in $\Omega$. 
\smallskip

Relying on that theorem, which may be seen as the extension of a result of Erokhin, proved in 1958 (see \cite{Erokhin};  see also 
Widom \cite{WID} which proved a more general result, with a different proof), to dimension $N > 1$, and as a result on the approximation of functions, we 
will give an application to the study of approximation numbers of a composition operator on $H^2 (\Omega)$ for a bounded symmetric domain of $\C^N$.
\goodbreak

%%%%%%%%%%%%%%%%%%%%%%%%%%%%%%%%%%%%%%%%%%%%%%%%%%%%%%%%%%%%%%%%%%%%%%%%%%%%%
\section{The spectral radius type formula}

%%%%%%%%%%%%%%%%%%%%%%%%%%%%%%%

In \cite[Section~6.2]{LIQR}, we proved the following result.
\begin{theorem}\label{kara}
Let $\phi \colon \D^N \to \D^N$ be given by $\phi (z_1, \ldots, z_N) = (r_1 z_1, \ldots, r_N z_N)$ where $0 < r_j < 1$.  Then:
\begin{displaymath} 
\beta_{N} (C_\phi) = \Gamma_N \big[ \overline{\phi (\D^N)} \big] = \Gamma_N \big[ \phi (\D^N)\big] \, .
\end{displaymath} 
\end{theorem}

The proof was simple, based on result of Blocki \cite{Blocki} on the Monge-Amp\`ere capacity of a cartesian product, and on the estimation, when $A \to \infty$, 
of the number $\nu_A$ of $N$-tuples $\alpha = (\alpha_1, \ldots, \alpha_N)$ of non-negative integers  $\alpha_j$ such that 
$\sum_{j = 1}^N \alpha_j \sigma_j \leq A$, where the numbers $\sigma_j > 0$ are fixed. The estimation was:
\begin{equation}\label{mata} 
\nu_A\sim \frac{A^N}{N ! \, \sigma_1 \cdots \sigma_N} \, \cdot 
\end{equation}
As J.~F.~Burnol pointed out to us, this is a consequence of the following elementary fact. Let $\lambda_N$ be the Lebesgue measure on $\R^N$, and let  $E$ be 
a compact subset of $\R^N$ such that $\lambda_{N} (\partial{E}) = 0$. Then:
\begin{displaymath} 
\lambda_N (E) = \lim_{A \to \infty} A^{- N} |(A \times E) \cap \Z^N| \, .
\end{displaymath} 
Then, just take $E = \{(x_1, \ldots, x_N) \, ; \ x_j \geq 0 \text{ and } \sum_{j = 1}^N x_j  \sigma_j \leq 1\}$.
\smallskip

In any case, this lets us suspect that the formula of Theorem~\ref{kara} holds in much more general cases. This is not quite true, as evidenced by our 
counterexample of \cite[Theorem~5.12]{LIQR}. Nevertheless, in good cases, this formula holds, as we will see in the next sections.

%%%%%%%%%%%%%%%%%%%%%%%%%%%
\medskip

In remaining of this section, we consider functions $\phi \colon \Omega \to \Omega$ such that $\overline{\phi (\Omega)} \subseteq \Omega$. If $\rho$ is an 
exhaustion function for $\Omega$, there is some $R_0 < 0$ such that $\overline{\phi (\Omega)} \subseteq B_\Omega (R_0)$, and that  implies that $C_\phi$ 
maps $H^2 (\Omega)$ into itself and is a compact operator (see \cite[Theorem~8.3]{Pol-Stess}, since, with their notations, for $r > R_0$, we have 
$T (r) = \emptyset$ and hence $\delta_\phi (r) = 0$) . 

\subsection{Minoration}

Recall that every hyperconvex domain $\Omega$ is pseudoconvex. By H.~Cartan-Thullen and Oka-Bremermann-Norguet theorems, being pseudoconvex is 
equivalent to being a domain of holomorphy, and equivalent to being holomorphically convex (meaning that if $K$ is a compact subset in $\Omega$, then its 
holomorphic hull $\tilde K$ is also contained in $\Omega$): see \cite[Corollaire~7.7]{Thiebaut}. Now (see \cite[Chapter~5, Exercise~11]{Krantz}, a domain of 
holomorphy $\Omega$ is said a \emph{Runge domain} if every holomorphic function in $\Omega$ can be approximated uniformly on its compact subsets by 
polynomials, and that is equivalent to saying that the polynomial hull and the holomorphic hull of every compact subset of $\Omega$ agree. By 
\cite[Chapter~5, Exercise~13]{Krantz}, every circled domain (in particular every bounded symmetric domain) is a Runge domain. 

\begin{definition}
A hyperconvex domain $\Omega$ is said \emph{strongly regular} if there exists a continuous {\it psh} exhaustion function $\rho$ such that all the sub-level sets:
\begin{displaymath} 
\Omega_c = \{ z \in \Omega \, ; \ \rho (z) < c \}
\end{displaymath} 
($c < 0$) have a regular closure. 
\end{definition}

For example, every bounded symmetric domain $\Omega$ is strongly regular since if $\| \, . \, \|$ is the associated norm, its sub-level sets $\Omega_c$ (with 
$\rho (z) = \log \| z \|$) are the open balls $B (0, \e^c)$, and the closed balls are regular, as said above.

\begin{theorem} \label{th mino}
Let $\Omega$ be a strongly regular bounded hyperconvex and Runge domain in $\C^N$, and let $\phi \colon \Omega \to \Omega$ be an analytic function 
such that $\overline{\phi (\Omega)} \subseteq \Omega$, and which is non-degenerate. Then:
\begin{equation} 
\Gamma_N \big[ \phi (\Omega) \big] \leq \beta_N^- (C_\phi) \, .
\end{equation} 
\end{theorem} 

Recall that if $\Omega$ is a domain in $\C^N$, a holomorphic function $\phi \colon \Omega \to \C^{M}$ ($M \leq N$) is \emph{non-degenerate} if there 
exists $a \in \Omega$ such that $\mathop{\mathrm {rank}}_a \phi = M$. Then $\phi (\Omega)$ has a non-empty interior.\par\smallskip

\begin{proof}
Let $(r_j)_{j \geq 1}$ be an increasing sequence of negative numbers tending to $0$. The set $H_j = \overline{\Omega_{r_j}}$ is a regular compact subset of 
$\Omega$, with non-void interior (hence non pluripolar). Let $\hat {H_j}$ its polynomial convex hull; this compact set is contained in $\Omega$, since $\Omega$ 
being a Runge domain, we have $\hat{H_j} = \tilde{H_j}$, and since $\tilde{H_j} \subseteq \Omega$, because $\Omega$ is holomorphically convex (being 
hyperconvex). Moreover $\hat {H_j}$ is regular since $V_E = V_{\hat E}$ for every compact subset of $\C^N$ (\cite[Corollary~4.14]{Siciak}).  \par
\smallskip

Let $K_j = \phi \big( \hat{H_j} \big)$ and let $G$ be a subspace of $H^2 (\Omega)$ with dimension $< n^N$.  \par\smallskip

The set $K_j$ is regular because of the following result (see \cite[Theorem~5.3.9]{Klimek}, \cite[top of page 40]{Plesniak}, 
\cite[Theorem~1.3]{Klimek-art2}, or \cite[Theorem~4]{Nguyen-Plesniak}, with a detailed proof).
\begin{theorem} [Ple{\'s}niak] 
Let $E$ be a compact, polynomially convex, regular and non pluripolar, subset of\/ $\C^N$. Then if $\Omega$ is a hyperconvex domain such that 
$E \subseteq \Omega$ and if $\phi \colon \Omega \to \C^N$ is a non-degenerate  holomorphic function, the set $\phi (E)$ is regular. 
\end{theorem}

As before, the polynomial convex hull $\hat{K_j}$ of $K_j$ is contained in $\Omega$ and is also regular. Since $\phi$ is non-degenerate, $K_j$ has a non-void 
interior; hence $\hat{K_j}$ also. We can hence use Zakharyuta's formula (Theorem~\ref{zak-niv}) for the compact set $\hat{K_j}$. 

By restriction, the subspace $G$ can be viewed as a subspace of ${\cal C} (\hat{K_j})$. By Zakharyuta's formula, 
for  $0 < \eps < 1$, there is $n_\eps \geq 1$ such that, for $n \geq n_\eps$:
\begin{displaymath} 
d_{n^N} (A_{\hat{K_j}}) \geq \exp \bigg[ - (1 + \eps) \, (2 \pi) \, n \, \bigg( \frac{N!}{\capa (\hat{K_j})}\bigg)^{1/N} \bigg] \, \cdot
\end{displaymath} 
Hence,  there exists $f \in B_{H^\infty} \subseteq B_{H^2}$ such that, for all $g \in G$:
\begin{displaymath} 
\| g - f \|_{{\cal C} (\hat{K_j})} \geq (1 - \eps)  \exp \bigg[ - (1 + \eps) \, (2 \pi) \, n \, \bigg( \frac{N!}{\capa (\hat{K_j})}\bigg)^{1/N} \bigg] \, \cdot
\end{displaymath} 
Since $\hat{K_j} = \tilde{K_j}$ and, by definition $\| \, . \, \|_{{\cal C} (\tilde{K_j})} = \| \, . \, \|_{{\cal C} (K_j)}$, we have:
\begin{displaymath} 
\| g - f \|_{{\cal C} (\hat{K_j})} = \| g - f \|_{{\cal C} (K_j)} = \| C_\phi (g) - C_\phi (f)\|_{{\cal C}(\tilde{H_j})} \, .
\end{displaymath} 
Equivalently, since, by definition $\| \, . \, \|_{{\cal C} (\tilde{H_j})} = \| \, . \, \|_{{\cal C} (H_j)}$, we have, for all $g \in G$:
\begin{displaymath} 
\| C_\phi (g) - C_\phi (f)\|_{{\cal C}(H_j)} \geq 
(1 - \eps) \exp \bigg[ - (1 + \eps) \, (2 \pi) \, n \, \bigg( \frac{N!}{\capa (\hat{K_j})}\bigg)^{1/N} \bigg] \, \cdot
\end{displaymath} 
This implies, thanks to \eqref{sup}, that, for all $g \in G$:
\begin{displaymath} 
\| C_\phi (g) - C_\phi (f)\|_{H^2 (\Omega)} \geq 
L_{r_j}^{- 1}  (1 - \eps) \exp \bigg[ - (1 + \eps) \, (2 \pi) \, n \, \bigg( \frac{N!}{\capa (\hat{K_j})}\bigg)^{1/N} \bigg]  \, \cdot
\end{displaymath} 
Using \eqref{alter}, we get, since the subspace $G$ is arbitrary:
\begin{displaymath} 
a_{n^N} (C_\phi) \geq L_{r_j}^{ - 1} (1 - \eps) \exp \bigg[ - (1 + \eps) \, (2 \pi) \, n \, \bigg( \frac{N!}{\capa (\hat{K_j})}\bigg)^{1/N} \bigg]  \, \cdot
\end{displaymath} 
Taking the $n$th-roots and passing to the limit, we obtain:
\begin{displaymath} 
 \beta_N^- (C_\phi) \geq \exp \bigg[ - (1 + \eps) \, (2 \pi) \, \bigg( \frac{N!}{\capa (\hat{K_j})}\bigg)^{1/N} \bigg]  \, \cdot
\end{displaymath} 
and then, letting $\eps$ go to $0$:
\begin{displaymath} 
\beta_N^- (C_\phi) \geq \exp \bigg[ - (2 \pi) \, \bigg( \frac{N!}{\capa (\hat{K_j})}\bigg)^{1/N} \bigg] = \Gamma_N (\hat{K_j}) \, .
\end{displaymath} 
Now, the sequence $(\hat{K_j})_{j \geq 1}$ is increasing and $\bigcup_{j \geq 1} \hat{K_j} \supseteq \phi (\Omega)$; hence, by 
\cite[Theorem~8.2~(8.3)]{BT},  we have 
$\capa (\hat{K_j}) \converge_{j \to \infty} \capa \big( \bigcup_{j \geq 1} \hat{K_j} \big) \geq \capa [\phi (\Omega)]$, so:
\begin{displaymath} 
\beta_N^- (C_\phi) \geq \Gamma_N [\phi (\Omega)] \, ,
\end{displaymath} 
and the proof of Theorem~\ref{th mino} is finished.
\end{proof} 
%

%%%%%%%%%%%%%%%%%%%%%%%%%%%%%%%%
\subsection{Majorization}

For the majorization, we assume different hypotheses on the domain $\Omega$. Nevertheless these assumptions agree with that of Theorem~\ref{th mino} when 
$\Omega$ is a bounded symmetric domain.

\subsubsection{Preliminaries} 

Recall that a domain $\Omega$ of $\C^N$ is a \emph{Reinhardt domain} (resp. \emph{complete Reinhardt domain}) if $z = (z_1, \ldots, z_N) \in \Omega$ 
implies that $( \zeta_1 z_1, \ldots, \zeta_N z_N) \in \Omega$ for all complex numbers $\zeta_1, \ldots, \zeta_N$ of modulus $1$ (resp. of modulus $\leq 1$). 
A complete bounded Reinhardt domain is hyperconvex if and only if $\log j_\Omega$ is 
{\it psh} and continuous in $\C^N \setminus \{0\}$, where $j_\Omega$ is the Minkowski functional of $\Omega$ (see 
\cite[Exercise following Proposition~3.3.3]{Blocki-Jag}). In general, the Minkowski functional $j_\Omega$ of a bounded complete Reinhardt domain $\Omega$ 
is {\it usc} and $\log j_\Omega$ is {\it psh} if and only if $\Omega$ is pseudoconvex (\cite[Theorem~1.4.8]{Blocki-Jag}). Other conditions for a bounded 
complete Reinhardt domain to being hyperconvex can found in \cite[Theorem~3.10]{Korean}. 

For a bounded hyperconvex and complete Reinhardt domain $\Omega$, its pluricomplex Green function with pole $0$ is $g_\Omega (z) = \log j_\Omega (z)$, 
where $j_\Omega$ is the Minkowski functional of $\Omega$ (\cite[Proposition~3.3.2]{Blocki-Jag}), and $S_\Omega (r) = \e^r \partial \Omega$. Since 
$\partial \Omega$ is in particular invariant by the pluri-rotations $z = (z_1, \ldots, z_N) \mapsto (\e^{i \theta_1} z_1, \ldots, \e^{i \theta_N} z_N)$, 
with $\theta_1, \ldots, \theta_N \in \R$, the harmonic measure $\tilde \mu_\Omega$ at $0$ (see the proof of Proposition~\ref{Hardy spaces}) is also invariant 
by the pluri-rotations (note that it is supported by the Shilov boundary of $\overline \Omega$: see \cite[very end of the paper]{Tiba}). We have, as in the proof 
of Proposition~\ref{Hardy spaces}, for $f \in H^2 (\Omega)$:
\begin{displaymath} 
\sup_{0 < s < 1} \int_{\partial \Omega} | f (s z) |^2 \, d \tilde \mu_\Omega (z) = \| f \|_{H^2 (\Omega)}^2 < \infty \, .
\end{displaymath} 
Since $\tilde \mu_\Omega$ is in particular invariant by the rotations $z \mapsto \e^{i \theta} z$, $\theta \in \R$, there exists, by \cite[Theorem~3]{Bochner},  
a function $f^\ast \in L^2 (\partial \Omega, \tilde\mu_\Omega)$ such that:
\begin{displaymath} 
\int_{\partial \Omega} |f (s z) - f^\ast (z) |^2 \, d\tilde\mu_\Omega (z) \converge_{s \to 1} 0 \, .
\end{displaymath} 
It ensues that the map $f \in H^2 (\Omega) \mapsto f^\ast \in L^2 (\partial \Omega, \tilde\mu_\Omega)$ is an isometric embedding (in fact, $f^\ast$ is the 
radial limit of $f$: see \cite[Lemma~2]{Hahn}). Therefore, we can consider $H^2 (\Omega)$ as a complemented subspace of 
$L^2 (\partial \Omega, \tilde\mu_\Omega)$, and we call $P$ the orthogonal projection of $L^2 (\partial \Omega, \tilde\mu_\Omega)$ onto $H^2 (\Omega)$. 
\smallskip

Every holomorphic function $f$ in a Reinhardt domain $\Omega$ containing $0$ (in particular if $\Omega$ is a complete Reinhardt domain) has a power series 
expansion about $0$:
\begin{displaymath} 
f (z) = \sum_{\alpha} b_\alpha z^\alpha 
\end{displaymath} 
which converges normally on compact subsets of $\Omega$ (\cite[Proposition~2.3.14]{Krantz}). Recall that if $z = (z_1, \ldots, z_N)$ and 
$\alpha = (\alpha_1, \ldots, \alpha_N)$, then $z^\alpha = z_1^{\alpha_1} \cdots z_N^{\alpha_N}$, $|\alpha| = \alpha_1 + \cdots + \alpha_N$, and 
$\alpha ! = \alpha_1! \cdots \alpha_N!$.

We have:
\goodbreak
\begin{proposition} \label{norme monomes}
Let $\Omega$ be a bounded hyperconvex and complete Reinhardt domain, and set $e_\alpha (z) = z^\alpha$. Then the system $(e_\alpha)_\alpha$ is orthogonal 
in $H^2 (\Omega)$.
\end{proposition}
\begin{proof}
We use the fact that the level sets $S (r)$ and the Demailly-Monge-Amp\`ere measures $\mu_r = \big(dd^c (g_\Omega)_r \big)^N$ are 
pluri-rotation invariant. For $\alpha \neq \beta$, we choose $\theta_1, \ldots, \theta_N \in \R$ such that $1, (\theta_1/ 2 \pi), \ldots, (\theta_N / 2 \pi)$ are 
rationally independent. Then $\exp \big[ i \big( \sum_{j = 1}^N (\alpha_j - \beta_j) \theta_j \big) \big] \neq 1$. Hence, as in \cite[p.~78]{Hua}, we have, 
making the change of variables $z = (\e^{i \theta_1} w_1, \ldots, \e^{i \theta_N} w_N)$:
\begin{displaymath} 
\int_{S (r)} z^\alpha \overline{z^\beta} \, d\mu_r (z) 
=  \exp \bigg[ i \bigg( \sum_{j = 1}^N (\alpha_j - \beta_j) \theta_j \bigg) \bigg]\int_{S (r)} w^\alpha \overline{w^\beta} \, d\mu_r (w) \, ,
\end{displaymath} 
which implies that:
\begin{displaymath} 
\int_{S (r)} z^\alpha \overline{z^\beta} \, d\mu_r (z) = 0 \, ,
\end{displaymath} 
and hence:
\begin{displaymath} 
(z^\alpha \mid z^\beta) := \lim_{r \to 0} \int_{S (r)} z^\alpha \overline{z^\beta} \, d\mu_r (z) = 0 \, . \qedhere
\end{displaymath} 
\end{proof}

For the polydisk, we have $\|e_\alpha \|_{H^2 (\D^N)} = 1$, and for the ball (see \cite[Proposition~1.4.9]{Rudin}):
\begin{displaymath} 
\| e_\alpha \|_{H^2 (\B_N)}^2 = \frac{(N - 1)! \, \alpha !}{(N - 1 + |\alpha|)!} \, \cdot
\end{displaymath} 
\begin{definition}
We say that $\Omega$ is a \emph{good} complete Reinhardt domain if, for some positive constant $C_N$ and some positive integer $c$, we have, 
for all $p \geq 0$:
\begin{displaymath} 
\sum_{|\alpha| = p} \frac{|z^\alpha|^2}{\| e_\alpha \|_{H^2 (\Omega)}^2} \leq C_N \, p^{c N} [j_\Omega (z)]^{2 p} \, ,
\end{displaymath} 
where $j_\Omega$ is the Minkowski functional of $\Omega$.
\end{definition}

\noindent{\bf Examples} \par

{\bf 1.} The polydisk $\D^N$ is a good Reinhardt domain because $\| e_\alpha \|_{H^2 (\D^N)} = 1$, $|z^\alpha| \leq \| z \|_\infty^{|\alpha|}$, 
and the number of indices $\alpha$ such that $|\alpha| = p$ is $\binom{N - 1 + p}{p} \leq C_N p^N$ (see \cite[p.~498]{Li-Queff} or 
\cite[pp.~213--214]{Li-Queff-Cambridge}).

{\bf 2.} The ball $\B_N$ is a good Reinhardt domain. In fact, observe that:
\begin{displaymath} 
\frac{(N - 1 + p)!}{(N - 1)!} = p! \, \frac{(p + 1) (p +2) \cdots (p+ N - 1)}{1 \times 2 \times \cdots \times (N - 1)} \leq p! \, (p + 1)^{N - 1} 
\leq p! \, (p + 1)^N \, ;
\end{displaymath} 
hence:
\begin{align*} 
\sum_{|\alpha| = p} \frac{|z^\alpha|^2\ }{\| e_\alpha \|_{H^2 (\B_N)}^2}
& = \sum_{|\alpha| = p} |z^\alpha|^2 \, \frac{(N - 1 + |\alpha|)!} {(N - 1)! \, \alpha !} \\
& \leq (p + 1)^N \sum_{|\alpha| = p} \frac{|\alpha|!}{\alpha !} \, |z_1|^{2 \alpha_1} \cdots |z_N|^{2 \alpha_N}  \\
& = (p + 1)^N (|z_1|^2 + \cdots + |z_N|^2)^p \, ,
\end{align*} 
by the  multinomial formula, so:
\begin{displaymath}
\sum_{|\alpha| = p} \frac{|z^\alpha|^2 \ }{\| e_\alpha \|_{H^2 (\B_N)}^2} \leq (p + 1)^N \| z \|_2^{2 p} \leq 2^N p^N \| z \|_2^{2 p} \, .
\end{displaymath}

{\bf 3.} More generally, if $\Omega = \B_{l_1} \times \cdots \times \B_{l_m}$, $l_1 + \cdots + l_m =N$, is a product of balls, we have, writing 
$\alpha = (\beta_1, \ldots, \beta_m)$, where each $\beta_j$ is an $l_j$-tuple:
\begin{align*} 
\| e_\alpha\|^2_{H^2 (\Omega)} 
& = \int_{{\mathbb S}_{l_1} \times \cdots \times {\mathbb S}_{l_2}} 
|u_1^{\beta_1}|^2 \ldots |u_m^{\beta_m}|^2 \, d\sigma_{l_1} (u_1) \ldots  d\sigma_{l_m} (u_m) \\
& \qquad \quad = \prod_{j = 1}^m \frac{(l_j - 1)! \, \beta_j !}{(l_j - 1 +|\beta_j|)!} \, \raise 1 pt \hbox{,}
\end{align*} 
and, writing $z = (z_1, \ldots, z_m)$, with $z_j \in \B_{l_j}$:
\begin{align*}
\sum_{|\alpha| = p} \frac{|z^\alpha|^2\ }{\| e_\alpha \|_{H^2 (\Omega)}^2} 
& \leq \sum_{p_1 + \cdots + p_m = p} \prod_{j = 1}^m (p_j + 1)^{l_j} \| z_j \|_2^{ 2 p_j} \\
& \leq C_m p^m \, (p + 1)^{l_1 + \cdots + l_m} [j_\Omega (z)]^{2 (p_1 + \cdots + p_m)} \, ,
\end{align*} 
since $j_\Omega (z) = \max \{ \| z_1 \|_2, \ldots, \| z_m\|_2\}$. Hence:
\begin{displaymath}
\sum_{|\alpha| = p} \frac{|z^\alpha|^2\ }{\| e_\alpha \|_{H^2 (\Omega)}^2} 
\leq C_N p^{2 N} [j_\Omega (z)]^{2 p} \, .
\end{displaymath}
%
 
%%%%%%%%%%%%%%%%%%%
\subsubsection{The result}
\begin{theorem} \label{th majo}
Let $\Omega$ be a bounded hyperconvex domain which is a good complete Reinhardt domain in $\C^N$, and let $\phi \colon \Omega \to \Omega$ be an 
analytic function such that $\overline{\phi (\Omega)} \subseteq \Omega$. Then, for every compact subset $K \supseteq \phi (\Omega)$ of $\Omega$ with 
non void interior, we have:
\begin{equation} 
\beta_N^+ (C_\phi) \leq \Gamma_N (K) \, .
\end{equation} 
In particular, if $\phi$ is moreover non-degenerate, we have:
\begin{equation} 
\beta_N^+ (C_\phi) \leq \Gamma_N \big[ \overline{\phi (\Omega)} \big] \, .
\end{equation} 
\end{theorem}

The last assertion holds because $\phi (\Omega)$ is open if $\phi$ is non-degenerate.
\begin{corollary}
Let $\Omega$ be a good complete bounded symmetric domain in $\C^N$, and $\phi \colon \Omega \to \Omega$ a non-degenerate analytic map such that 
$\overline{\phi (\Omega)} \subseteq \Omega$. Then:
\begin{displaymath} 
\Gamma_N \big[ \phi (\Omega) \big] \leq \beta_N^- (C_\phi) \leq \beta_N^+ (C_\phi) \leq \Gamma_N \big[ \overline{\phi (\Omega)} \big] \, .
\end{displaymath} 
\end{corollary}

For the proof of Theorem~\ref{th majo}, we will use the following result (\cite[Proposition~6.1]{ZAK}), which do not need any regularity condition on the 
compact set (because it may be written as a decreasing sequence of regular compact sets).
\begin{proposition} [Zakharyuta] \label{Zakha-upper}
If $K$ is any compact subset of a bounded hyperconvex domain $\Omega$ of $\C^N$ with non-empty interior, we have:
\begin{displaymath} 
\limsup_{n \to \infty} \frac{\log d_n (A_K)}{n^{1/N}} \leq - \bigg( \frac{N!}{\tau_N (K)} \bigg)^{1/N} \, .
\end{displaymath} 
\end{proposition}
\begin{proof} [Proof of Theorem~\ref{th majo}]
In the sequel we write $\| \, . \,\|_{H^2}$ for $\| \, . \,\|_{H^2 (\Omega)}$.
We set:
\begin{displaymath} 
\Lambda_N = \limsup_{n \to \infty} [ d_n (A_K) ]^{n^{- 1/ N}} \, .
\end{displaymath} 

Changing $n$ into $n^N$, Proposition~\ref{Zakha-upper} means that for every $\eps > 0$, there exists, for $n$ large enough, an $(n^N - 1)$-dimensional 
subspace $F$ of $\mathcal{C}(K)$ such that, for any $g \in H^{\infty} (\Omega)$, there exists $h \in F$ such that:
\begin{equation} \label{demi}
\Vert g - h\Vert_{\mathcal{C}(K)} \leq (1 + \eps)^n \Lambda_N^n \, \Vert g \Vert_\infty \, .  
\end{equation} 

Let $l$ be an integer to be adjusted later, and  
\begin{displaymath} 
f (z) = \sum_{\alpha} b_\alpha z^\alpha \in H^2 (\Omega) \quad \text{with } \| f \|_{H^2} \leq 1 \, .
\end{displaymath} 
By Proposition~\ref{norme monomes}, we have:
\begin{displaymath} 
\| f \|_{H^2}^2 = \sum_\alpha |b_\alpha|^2 \|e_\alpha \|_{H^2}^2 \, .
\end{displaymath} 
We set:
\begin{displaymath} 
g (z) = \sum_{|\alpha|\leq l} b_\alpha z^\alpha \, . 
\end{displaymath} 
By the Cauchy-Schwarz inequality:
\begin{displaymath} 
| g (z)|^2 \leq \bigg( \sum_{|\alpha|\leq l} |b_\alpha|^2 \| e_\alpha \|_{H^2}^2 \bigg) 
\bigg( \sum_{|\alpha|\leq l} \frac{|z^{\alpha}|^2 \ }{\| e_\alpha \|_{H^2}^2} \bigg) 
\leq \sum_{|\alpha|\leq l} \frac{|z^{\alpha}|^2 \ }{\| e_\alpha \|_{H^2}^2}\, \cdot
\end{displaymath} 
Since $\Omega$ is a good complete Reinhardt domain and since $j_\Omega (z) < 1$ for $z \in \Omega$, we have:
\begin{displaymath} 
| g (z)|^2 \leq \sum_{p = 0}^l p^{c N}  [j_\Omega (z)]^{2 p} \leq (l + 1)^{c N + 1} \, . 
\end{displaymath} 
It follows from \eqref{demi} that there exists $h \in F$ such that:
\begin{displaymath} 
\| g - h \|_{\mathcal{C}(K)} \leq (1 + \eps)^n \Lambda_N^n (l + 1)^{(c N + 1)/2} \, .  
\end{displaymath} 

Since $C_\phi f (z) - C_\phi \, g (z) = f \big( \phi (z) \big) - g \big( \phi (z) \big)$ and $\overline{\phi (\Omega)} \subseteq K$, we have 
\hbox{$\| C_\phi f - C_\phi g \|_\infty \leq \| f - g \|_{{\cal C} (K)}$}; therefore:
\begin{equation} \label{demi+}
\begin{split}
\| g \circ \phi - h \circ \phi \|_{H^2} 
& \leq \| g \circ \phi - h \circ \phi \|_{\infty} \leq \| g - h \|_{\mathcal{C}(K)} \\ 
& \leq (1 + \eps)^n \Lambda_N^n (l + 1)^{(c N + 1)/2} \, .  
\end{split}
\end{equation} 

Now, the subspace $\tilde F$ formed by functions $v \circ \phi$, for $v \in F$, can be viewed as a subspace of 
$L^{\infty}(\partial \Omega, \tilde \mu_\Omega) \subseteq L^{2}(\partial \Omega, \tilde \mu_\Omega)$ (indeed, since $v$ is continuous, we can write 
$(v \circ \phi)^\ast = v \circ \phi^\ast$, where $\phi^\ast$ denotes the almost everywhere existing radial limits of $\phi (r z)$, which belong to $K$). Let 
finally $E = P (\tilde F) \subseteq H^2 (\Omega)$ where $P \colon L^{2}(\partial \Omega, \tilde \mu_\Omega) \to H^2 (\Omega)$ is the orthogonal 
projection. This is a subspace of $H^2 (\Omega)$ with dimension $< n^N$, and we have 
${\rm dist}\, (C_\phi g, E) \leq \| g \circ \phi - P (h \circ \phi) \|_{H^2}$; hence, by \eqref{demi+}:
\begin{equation} \label{demi++} 
{\rm dist}\, (C_\phi g, E) \leq (1 + \eps)^n \Lambda_N^n (l + 1)^{(c N + 1)/2} \, . 
\end{equation} 

Now, the same calculations give that:
\begin{displaymath} 
| f (z) - g (z) |^2 \leq \sum_{p > l} p^{c N} [j_\Omega (z)]^{2 p} \, ;
\end{displaymath} 
hence, for some positive constant $M_N$: 
\begin{displaymath} 
| f (z) - g (z) | \leq M_N (l+1)^{(c N + 1)/2} \frac{[j_\Omega (z)]^l \ }{(1 - [j_\Omega (z)]^2)^{(c N + 1)/2}} \, \raise 1 pt \hbox{,}
\end{displaymath} 
by using the following lemma, whose proof is postponed.
\begin{lemma} \label{lemme utile}
For every non-negative integer $m$, there exists a positive constant $A_m$ such that, for all integers $l \geq 0$ and all $0 < x < 1$, we have:
\begin{displaymath} 
\sum_{p\geq l} p^{m} x^p \leq A_m l^m \frac{x^{l} \,}{(1-x)^{m+1}} \, .
\end{displaymath} 
\end{lemma}

Since $K$ is a compact subset of $\Omega$, there is a positive number $r_0 < 1$ such that $j_\Omega (z) \leq r_0$ for $z \in K$. 
Since $C_\phi f (z) - C_\phi g (z) = f \big( \phi (z) \big) - g \big( \phi (z) \big)$ and $\overline{\phi (\Omega)} \subseteq K$, we have 
$\| C_\phi f - C_\phi g \|_\infty \leq \| f - g \|_{{\cal C} (K)}$, and we get: 
\begin{equation}  \label{huit et demi}
\| C_\phi f - C_\phi g \|_{H^2} \leq \| C_\phi f - C_\phi g \|_\infty \leq M_N \, (l + 1)^{(c N + 1)/2} \, \frac{r_0^l}{(1 - r_0^2)^{(c N + 1)/2}} \, \cdot
\end{equation} 
Now, \eqref{demi++} and \eqref{huit et demi} give:
\begin{displaymath} 
{\rm dist}\, (C_\phi f, E) \leq M_N \, (l + 1)^{(c N +1 )/2} \,\bigg( \frac{r_0^l}{(1 - r_0^2)^{(c N + 1)/2}} + (1 + \eps)^n \Lambda_N^n \bigg) \, .
\end{displaymath}
It ensues, thanks to \eqref{a=d}, that:
\begin{displaymath} 
\big[ a_{n^N} (C_\phi) \big]^{1/n}
\leq  [M_N \, (l + 1)^{(c N + 1)/2}]^{1/n} \, \bigg[ \frac{r_0^{l/n}}{(1 - r_0^2)^{(c N + 1)/2n}} + (1 + \eps) \, \Lambda_N \bigg] \, .
\end{displaymath} 
Taking now for $l$ the integer part of $n \log n$, and passing to the upper limit as $n \to \infty$, we obtain (since $l / n \to \infty$ and $(\log l)/n\to 0$):
\begin{displaymath} 
\beta_N^{+} (C_\phi) \leq (1 + \eps) \, \Lambda_N \, ,
\end{displaymath} 
and therefore, since $\eps > 0$ is arbitrary:
\begin{displaymath} 
\beta_N^{+} (C_\phi) \leq \Lambda_N \, . 
\end{displaymath} 
That ends the proof, by using Proposition~\ref{Zakha-upper}.
\end{proof}
\begin{proof} [Proof of Lemma~\ref{lemme utile}]
We make the proof by induction on $m$. We set:
\begin{displaymath} 
S_m = \sum_{p\geq l} p^{m} x^p
\end{displaymath} 
The result  is obvious for $m = 0$, with $A_0 = 1$, since then
$S_0 = \sum_{p \geq l} x^p =\frac{x^l}{1 - x} \, \cdot$
Let us assume that it holds till $m - 1$ and prove it for $m$. We observe that, since $p^m - (p - 1)^{m}\leq m p^{m - 1}$, we have:
\begin{align*}
(1 - x) S_m 
& =\sum_{p\geq l} p^{m}x^p - \sum_{p \geq l} p^{m} x^{p + 1} =\sum_{p \geq l} p^{m} x^p - \sum_{p \geq l+1} (p - 1)^{m} x^{p} \\
& =\sum_{p \geq l + 1} (p^m - (p - 1)^{m}) x^p + l^m x^l \leq \sum_{p \geq l + 1} m p^{m - 1}x^p + l^m x^l \\ 
& \leq\sum_{p \geq l} m p^{m - 1}x^p + l^m x^l \leq m A_{m - 1} l^{m - 1} \frac{x^{l}}{(1 - x)^m} + l^m x^l \\
& \leq (m A_{m - 1} + 1) \, l^{m}\frac{x^{l}}{(1 - x)^m} \, \raise 1 pt \hbox{,}
\end{align*}
giving the result, with $A_m = m A_{m - 1} + 1$. 
\end{proof}
\goodbreak
%%%%%%%%%%%%%%%%%%%%%%%%%
\subsection{Equality}
\begin{proposition} \label{capa adh}
Let $\Omega$ be a bounded hyperconvex domain and $\omega$ a relatively compact open subset of $\Omega$. Assume that:
\begin{equation} \label{eq ast}
\begin{split}
& \text{For every $a \in \partial \omega$, except on a pluripolar set $E \subseteq \partial \omega$, there exists}  \\  
&  \text{$z_0 \in \omega$ such that the open segment $(z_0, a)$ is contained in $\omega$.} 
\end{split}
\end{equation}

Then:
\begin{displaymath} 
\capa (\overline \omega) = \capa (\omega) \, .
\end{displaymath} 

In particular, if  $\phi \colon \Omega \to \Omega$ a non-degenerate holomorphic map such that $\overline{\phi (\Omega)} \subseteq \Omega$ and 
$\omega = \phi (\Omega)$ satisfies \eqref{eq ast}, we have:
\begin{displaymath} 
\capa \big[\phi (\Omega) \big] = \capa \big[ \overline{\phi (\Omega)} \big] \, .
\end{displaymath} 
\end{proposition}

Before proving Proposition~\ref{capa adh}, let us give an example of such a situation.
\begin{proposition} 
Let $\Omega$ be a bounded hyperconvex domain with $C^1$ boundary. Let $U$ be an open neighbourhood of $\overline{\Omega}$ and 
$\phi \colon U \to \C^N$ be a non-degenerate holomorphic function such that $\overline{\phi (\Omega)} \subseteq \Omega$. Then the condition \eqref{eq ast} 
is satisfied.
\end{proposition}
\begin{proof} 
Let $\omega = \phi (\Omega)$. 

We may assume that $U$ is connected, hyperconvex and bounded. Let $B_\phi$ be the set of points $z \in U$ such that the complex Jacobian 
$J_\phi$ is null. Since $J_\phi$ is holomorphic in $\Omega$, we have $\log |J_\phi| \in {\cal PSH} (U)$ and hence (see \cite[proof of Lemma~10.2]{KOSK}):
\begin{displaymath} 
B_\phi = \{ z \in U \, ; \ J_\phi (z) = 0 \} = \{z \in U \, ; \ \log |J_\phi (z)| = - \infty\} 
\end{displaymath} 
is pluripolar . Therefore (see \cite[Theorem~6.9]{BT}), $\capa (B_\phi, U) = 0$. It follows (see \cite[page~2, line -8]{BT}) that 
$\capa [ \phi (B_\phi) ] := \capa [ \phi (B_\phi), \Omega ] = 0$.
\par\smallskip

Now, for every $a \in \partial {\overline \omega} \cap [\phi (U \setminus B_\phi)]$, there is a tangent hyperplane $H_a$ to $\overline \omega$, 
and hence an inward normal to $\partial \overline \omega$ (note that  $\partial {\overline \omega} \subseteq \phi (\partial \Omega) \subseteq \phi (U)$). 
It follows that there is $z_0 \in \omega$ such that the open interval $(z_0, a)$ is contained in $\omega$. 
\end{proof}
\begin{proof} [Proof of Proposition~\ref{capa adh}]  
Let $a \in \partial \omega$ and $L$ be a complex line containing $(z_0, a)$; we have $a \in \overline{\omega \cap L}$. 
Assume now that this point $a$ is a \emph{fine} (``\emph{effil\'e}'') point of $\omega$, i.e. that there exists $u \in {\cal PSH} (V)$,  
for $V$ a neighbourhood of $a$, such that:
\begin{displaymath} 
\limsup_{z \to a\, , z \in \omega} u (z) < u (a) \, .
\end{displaymath} 
By definition, the restriction $\tilde u$ of $u$ to $\omega \cap L$ is subharmonic and we keep the inequality:
\begin{displaymath} 
\limsup_{z \to a\, , z \in \omega \cap L} {\tilde u} (z) < {\tilde u} (a) = u (a) \, .
\end{displaymath} 
That means that $a$ is a fine point of $\omega \cap L$. But $a \in \overline{\omega \cap L}$ and $\omega \cap L$ is connected, so this is not possible, by 
\cite[Lemma~2.4]{LQR}. Hence no point of $\partial \omega \setminus E$ is fine.
\par\smallskip

Let now $\omega^f$ be the closure of $\omega$ for the fine topology (i.e. the coarsest topology on $U$ for which all the functions in ${\cal PSH} (U)$ are 
continuous; it is known: see \cite[comment after Theorem~2.3]{BT2}, that it is the trace on $U$ of the fine topology on $\C^N$). It is also known 
(see \cite[Corollary~4.8.10]{Klimek}) that $\omega^f$ is the set of points of $\overline{\omega}$ which are not fine. By the above reasoning, we thus have:
\begin{displaymath} 
\overline \omega \setminus \omega^f \subseteq E \, .
\end{displaymath} 
Since $\capa (E) = 0$, we have:
\begin{displaymath} 
\capa (\overline \omega \setminus\omega^f ) = 0 \, ,
\end{displaymath} 
and it follows that:
\begin{displaymath} 
\capa (\overline \omega) = \capa [ \omega^f \cup (\overline \omega \setminus\omega^f ) ] 
\leq \capa (\omega^f ) + \capa (\overline \omega \setminus\omega^f ) = \capa (\omega^f ) \, ,
\end{displaymath} 
and hence $\capa (\omega^f) = \capa (\overline \omega)$.

But, since, by definition, the \emph{psh} functions are continuous for the fine topology, it is clear, that the relative extremal functions $u_{\omega, \Omega}$ 
and $u_{\omega^f, \Omega}$ are equal; hence we have, by \cite[Proposition~4.7.2]{Klimek}:
\begin{displaymath} 
\capa (\omega) = \int_\Omega (dd^c u_{\omega, \Omega}^\ast)^N = \int_\Omega (dd^c u_{\omega^f, \Omega}^\ast)^N = \capa (\omega^f) \, .
\end{displaymath} 
Hence $\capa (\omega) = \capa (\overline\omega)$.
\end{proof}
%

%%%%%%%%%%%%%%%%%%%%%%%%%%%%%%%%%%%%%%%%%%%%%%%%%%%%%%%%%%%%%%%%%%%%%%%%%%%%%
\subsection{Consequences of the spectral radius type formula}

Theorem~\ref{th mino} has the following consequence.
\begin{proposition} \label{capa infinie}
Let $\Omega$ be a regular bounded symmetric domain in $\C^N$, and let $\phi \colon \Omega \to \Omega$ be a non-degenerate analytic 
function inducing a bounded composition operator $C_\phi$ on $H^2 (\Omega)$. 

Then, if $\capa [\phi (\Omega)] = \infty$, we have $\beta_N (C_\phi) = 1$. 

In other words, if, for some constants $C, c > 0$, we have $a_n (C_\phi) \leq C\, \e^{- c n^{1/N}}$ for all $n \geq 1$, then 
$\capa [\phi (\Omega)] < \infty$.
\end{proposition}

As a corollary, we can give a new proof of \cite[Theorem~3.1]{LIQR}.
\begin{corollary}
Let $\tau \colon \D \to \D$ be an analytic map such that $\| \tau \|_\infty = 1$ and $\psi \colon \D^{N - 1} \to \D^{N - 1}$ such that the map 
$\phi \colon \D^N \to \D^N$, defined as:
\begin{displaymath} 
\phi (z_1, z_2, \ldots, z_N) = \big( \tau (z_1), \psi (z_2, \ldots, z_N) \big) \, , 
\end{displaymath} 
is non-degenerate. Then $\beta_N (C_\phi) = 1$.
\end{corollary}
\begin{proof}
Since the map $\phi$ is non-degenerate, $\psi$ is also non-degenerate. Hence (see \cite[Proposition~2]{Nguyen-Plesniak} $\psi (\D^{N  - 1})$ is not pluripolar, 
i.e. ${\rm Cap}_{N - 1} [\psi (\D^{N - 1})] > 0$. On the other hand, it follows from \cite[Theorem~3.13 and Theorem~3.14]{LQR} that 
${\rm Cap}_1 [\tau (\D)] = + \infty$. Then, by \cite[Theorem~3]{Blocki}, we have:
\begin{align*} 
{\rm Cap}_N [\phi (\D^N)] 
& = {\rm Cap}_N [\tau (\D)  \times \psi (\D^{N - 1})]  \\
& = {\rm Cap}_1 [\tau (\D)] \times {\rm Cap}_{N  - 1} [\psi (\D^{N - 1})] = + \infty \, . 
\end{align*} 
It follows from Proposition~\ref{capa infinie} that $\beta_N (C_\phi) = 1$.
\end{proof} 
\begin{proof} [Proof of Proposition~\ref{capa infinie}]
If $R \colon H^2 (\Omega) \to H^2 (\Omega)$ is a finite-rank operator, we set, for $t < 0$:
\begin{displaymath} 
\qquad \quad (R_t f) (w) = (R f) (\e^t w) \, , \qquad f \in H^2 (\Omega) \, .
\end{displaymath} 
Then the rank of the operator $R_t$ is less or equal to that of $R$.

Recall that if $\| \, . \, \|$ is the norm whose unit ball is $\Omega$, then the pluricomplex Green function of $\Omega$ is $g_\Omega (z) = \log \| z \|$, 
and hence the level set $S(r)$ is the sphere $S (0, \e^r) = \e^r \partial \Omega$ for this norm. Since:
\begin{displaymath} 
\int_{S (r)} | f [ \phi (\e^t w)] - (R f) (\e^t w)|^2 \, d \mu_r (w) = \int_{S (r + t)} |f [\phi (z)] - (R f) (z) |^2 \, d\mu_{r + t} (z) \, ,
\end{displaymath} 
we have, setting $\phi_t (w) = \phi (\e^t w)$:
\begin{displaymath} 
\| C_{\phi_t} (f) - R_t (f) \|_{H^2} \leq \| C_\phi (f) - R (f)\|_{H^2} \, .
\end{displaymath} 
It follows that $a_n (C_{\phi_t}) \leq a_n (C_\phi)$ for every $n \geq 1$. Therefore $\beta_N^- (C_{\phi_t}) \leq \beta_N^- (C_\phi)$. 

By Theorem~\ref{th mino}, we have:
\begin{displaymath} 
\exp \Bigg[ - 2 \pi \bigg( \frac{N!}{\capa [\phi_t ( \Omega) ]} \bigg) ^{1/N}\Bigg]  \leq \beta_N^- (C_{\phi_t}) \, .
\end{displaymath} 
Since $\phi_t (\Omega) = \phi (\e^t \Omega)$ increases to $\phi (\Omega)$ as $t \uparrow 0$, we have (see \cite[Corollary~4.7.11]{Klimek}):
\begin{displaymath} 
\capa [\phi (\Omega)] = \lim_{t \to 0} \capa [\phi_t (\Omega)] \,. 
\end{displaymath} 
As $\capa [\phi (\Omega)] = \infty$, we get:
\begin{displaymath} 
\beta_N^- (C_\phi) \geq \limsup_{t \to 0} \beta_N^- (C_{\phi_t}) = 1 \, . \qedhere
\end{displaymath} 
\end{proof}

\noindent{\bf Remark 1.} In \cite[Theorem~5.12]{LIQR}, we construct a non-degenerate analytic function $\phi \colon \D^2 \to \D^2$ such that 
$\overline{\phi (\D^2)} \cap \partial \D^2 \neq \emptyset$ and for which $\beta_2^+ (C_\phi) < 1$. We hence have $\capa [\phi (\D^2)] < \infty$. 
\smallskip

\noindent{\bf Remark 2.} The capacity cannot tend to infinity too fast when the compact set approaches the boundary of $\Omega$; in fact, we have the 
following result, that we state for the ball, but which holds more generally.
\begin{proposition}
For every compact set $K$ of $\B_N$, we have, for some constant $C_N$:
\begin{displaymath} 
\capa (K) \leq \frac{C_N}{[{\rm dist}\, (K, {\mathbb S}_N)]^N} \, \cdot
\end{displaymath} 
\end{proposition}
\begin{proof}
We know that:
\begin{displaymath} 
\capa (K) = \int_{\B_N} (dd^c u_K^\ast)^N \, .
\end{displaymath} 
Let $\rho (z) = |z|^2 - 1$ and $a_K := \min_{z \in K} [- \rho (z)] = - \max _{z \in K} \rho (z)$. Then $\rho$ is in ${\cal PSH}$ and is non-positive. 
Since $a_K > 0$, the function:
\begin{displaymath} 
v (z) = \frac{\rho (z)}{a_K} 
\end{displaymath} 
is in ${\cal PSH}$, non-positive on $\B_N$, and $v \leq - 1$ on $K$. Hence $v \leq u_K \leq u_K^\ast$. 

Since $v (w) = 0$ for all $w \in {\mathbb S}_N$ and (see \cite[Proposition~6.2~(iv)]{BT}, or \cite[Proposition~4.5.2]{Klimek}):
\begin{displaymath} 
\lim_{z \to w} u_K^\ast (z) = 0 \, ,
\end{displaymath} 
for all $w \in {\mathbb S}_N$, the comparison theorem of Bedford and Taylor (\cite[Theorem~4.1]{BT}; \cite[Theorem~3.7.1]{Klimek} gives, since 
$v \leq u_K^\ast$ and $v, u_K^\ast \in {\cal PSH}$:
\begin{displaymath} 
\int_{\B_N} (dd^c u_K^\ast)^N \leq \int_{\B_N} (dd^c v)^N = \frac{1}{a_K^N} \int_{\B_N} (dd^c \rho)^N \, .
\end{displaymath} 
As $(dd^c \rho)^N = 4^N N! \, d\lambda_{2N}$, we get, with $C_N := 4^N N! \, \lambda_{2N} (\B_N)$:
\begin{displaymath} 
\capa (K) \leq \frac{C_N}{a_K^N}\, \cdot
\end{displaymath} 
That ends the proof since:
\begin{displaymath} 
a_K = \min_{z \in K} (1 - |z|^2) \geq \min_{z \in K} (1 - |z|) = {\rm dist}\, (K, {\mathbb S}_N) \, \qedhere
\end{displaymath} 
\end{proof}

We have assumed that the symbol $\phi$ is non-degenerate. For a degenerate symbol $\phi$, we have:
\begin{proposition}
Let $\Omega$ be a bounded hyperconvex and good complete Reinhardt domain in $\C^N$, and let $\phi \colon \Omega \to \Omega$ be an  
analytic function such that $\overline{\phi (\Omega)} \subseteq \Omega$ is pluripolar. Then $\beta_N (C_\phi) = 0$.
\end{proposition}
Recall that $\phi (\Omega)$ is pluripolar when $\phi$ is degenerate (see \cite[Proposition~2]{Nguyen-Plesniak}); its closure is also pluripolar if it satisfies  
the condition \eqref{eq ast}.
\begin{proof}
Let $K = \overline{\phi (\Omega)}$. By hypothesis, we have $\capa (K) = 0$. For every $\eps > 0$, let 
$K_\eps = \{ z \in \Omega \, ; \ {\rm dist}\, (z, K) \leq \eps\}$. By Theorem~\ref{th majo}, we have $\beta_N^+ (C_\phi) \leq \Gamma_N (K_\eps)$. 
As $\lim_{\eps \to 0} \capa (K_\eps) = \capa (K) = 0$ (\cite[Proposition~4.7.1$(iv)$]{Klimek}), we get $\beta_N (C_\phi) = 0$. 
\end{proof}

\noindent{\bf Remark~1.} In \cite[Section~4]{LIQR}, we construct a degenerate symbol $\phi$ on the bi-disk $\D^2$, defined by 
$\phi (z_1, z_2) = \big( \lambda_\theta (z_1), \lambda_\theta (z_1) \big)$, where $\lambda_\theta$ is a lens map, for which $\beta^- (C_\phi) > 0$. For 
this function $\overline{\phi (\D^2)} \cap \partial \D^2 \neq \emptyset$ and hence $\overline{\phi (\D^2)}$ is not a compact subset of $\D^2$. 
\smallskip

\noindent {\bf Remark~2.} In the one dimensional case, for any (non constant) analytic map $\phi \colon \D \to \D$, the parameter 
$\beta (C_\phi) = \beta_1 (C_\phi)$ is determined by its range $\phi (\D)$, as shown by the formula:
\begin{displaymath}
\beta (C_\phi) = \e^{- 1/ \capa [\phi (\D)]}
\end{displaymath}
proved in \cite{LQR}. This is no longer true in dimension $N \geq 2$. In \cite{LIQR-surjective}, we construct pairs of (degenerate) symbols 
$\phi_1, \phi_2 \colon \D^2 \to \D^2$, such that $\phi_1 (\D^2)= \phi_2 (\D^2)$ and:
\smallskip

1) $C_{\phi_1}$ is not bounded, but $C_{\phi_2}$ is compact, and even $\beta_2 (C_{\phi_2}) = 0$; 
\smallskip

2) $C_{\phi_1}$ is bounded but not compact, so $\beta_2 (C_{\phi_1}) = 1$, and $C_{\phi_2}$ is compact, with $\beta_2 (C_{\phi_2}) = 0$; 
\smallskip

3) $C_{\phi_1}$ is compact, with $0 < \beta_2 (C_{\phi_1}) < 1$, and $C_{\phi_2}$ is compact, with $\beta_2 (C_{\phi_2}) = 0$.

%%%%%%%%%%%%%%%%%%%%%%%%%%%%%%%%%%%%%%%%%%%%%%%%%%%%%%%%%%%%%%%%%%%%%%%%
\bigskip

\noindent{\bf Acknowledgements.} We thank S.~Nivoche and A.~Zeriahi for useful discussions and informations, and Y.~Tiba, who send 
us his paper \cite{Tiba}. We than specially S.~Nivoche, who carefully read a preliminary version of this paper. \par
\smallskip

The third-named author is partially supported by the project MTM2015-63699-P (Spanish MINECO and FEDER funds).

\bigskip

%%%%%%%%%%%%%%%%%%%%%%%%%%%%%%%%%%%%%%%%%%%%%%%%%%%%%%%%%%%%%%%%%%%%%%

%%%%%%%%%%%%%%%%%%%%%%%%%%%%%%%%%%%%%%%%%%%%%%%%%%%%%%%%%%%%%%%%%%%%%%%%%
%\newpage 
\smallskip

{\footnotesize
Daniel Li \\ 
Univ. Artois, Laboratoire de Math\'ematiques de Lens (LML) EA~2462, \& F\'ed\'eration CNRS Nord-Pas-de-Calais FR~2956, 
Facult\'e Jean Perrin, Rue Jean Souvraz, S.P.\kern 1mm 18 
F-62\kern 1mm 300 LENS, FRANCE \\
daniel.li@euler.univ-artois.fr
\smallskip

Herv\'e Queff\'elec \\
Univ. Lille Nord de France, USTL,  
Laboratoire Paul Painlev\'e U.M.R. CNRS 8524 \& F\'ed\'eration CNRS Nord-Pas-de-Calais FR~2956 
F-59\kern 1mm 655 VILLENEUVE D'ASCQ Cedex, FRANCE \\
Herve.Queffelec@univ-lille1.fr
\smallskip
 
Luis Rodr{\'\i}guez-Piazza \\
Universidad de Sevilla, Facultad de Matem\'aticas, Departamento de An\'alisis Matem\'atico \& IMUS,  
Calle Tarfia s/n \\ 
41\kern 1mm 012 SEVILLA, SPAIN \\
piazza@us.es
}

\end{document}